\documentclass[preprint,12pt]{elsarticle}
\usepackage{latexsym}
\usepackage{amssymb}
\usepackage{amsmath}
\usepackage{amsthm}
\usepackage{verbatim}
\usepackage[pagewise]{lineno}
\numberwithin{equation}{section}
\newtheorem{theorem}{Theorem}[section]
\newtheorem{lemma}{Lemma}[section]

\allowdisplaybreaks
\usepackage{geometry}
\biboptions{sort&compress}

\begin{document}
\begin{frontmatter}



\title{A new blowup criterion for strong solutions of a coupled periodic Camassa-Holm system
}


\author[ad1]{Yonghui Zhou\corref{cor}}
\ead{zhouyhmath@163.com}
\author[ad2]{Xiaowan Li}
\ead{xiaowan0207@163.com}

\address[ad1]{School of Mathematics, Hexi University, Zhangye 734000, P.R. China}
\address[ad2]{College of Mathematics and System Sciences, Xinjiang University, Urumqi, 830046, P.R. China}

\cortext[cor]{Corresponding author.}

\begin{abstract}
This paper is concerned with the wave breaking phenomena for a coupled periodic Camassa-Holm system. We establish a new blowup criterion for strong solutions by the method of characteristic and convolution estimates, and also give the existence interval of the blowup point.

\end{abstract}


\end{frontmatter}


\section{Introduction}
In this paper, we consider the following Cauchy problem of a coupled periodic Camassa-Holm system
\begin{equation}
\begin{cases}
m_{t}+2m u_{x}+m_{x}u+(mv)_{x}+n v_{x}=0,\\
n_{t}+2n v_{x}+n_{x}v+(nu)_{x}+m u_{x}=0,
\end{cases}
\label{101}
\end{equation}
subject to the initial data
\begin{equation}
\begin{cases}
u(0,x)=u_{0}(x),
v(0,x)=v_{0}(x),\\
u(t,x+1)=u(t,x),
v(t,x+1)=v(t,x),
\end{cases}
\label{101a}
\end{equation}
where $m=u-u_{xx}$ and $n=v-v_{xx}$. Such a model was firstly studied by Fu et al. \cite{Fu2009,Fu2010} as a multi-component generalization of the Camassa-Holm equation with peakons in the form of a superposition of multi-peakons.

If $v=0$, \eqref{101} corresponds to the following classical Camassa-Holm (CH) equation
\begin{equation}
m_{t}+2m u_{x}+m_{x}u=0,
\label{102a}
\end{equation}
which was was derived as a bi-Hamiltonian, shallow water equation in \cite{Camassa1993}. Afterwards, alternative derivations of \eqref{102a} as a model for shallow water waves of moderate amplitude have been obtained by Johnson \cite{Johnson2002} and Constantin
and Lannes \cite{Constantin2009}. A very interesting physical feature of CH is that it models breaking waves, that is, it admits solutions that develop singularities in finite time: they remain bounded, but their slopes become unbounded \cite{Constantin1998} (also see \cite{Beals2000}). The CH equation also admits peaked traveling waves that interact like solitons \cite{Beals2000,Camassa1993,Lenells2005} and, moreover, these peaked waves are orbitally stable: their shapes are stable under small perturbations and
therefore they can be recognized physically \cite{Constantin20003,Lenells2004}.


The multi-component Camassa-Holm equations have been widely studied in recent years. The most popular one among them is the following two-component Camassa-Holm system \cite{Olver1996}
\begin{equation}
\begin{cases}
m_{t}+um_{x}+2mu_{x}+vv_{x}=0,\\
v_{t}+(uv)_{x}=0,
\end{cases}
\label{103}
\end{equation}
where $m=u-u_{xx}$, $u(t,x)$ has the same meaning as in equation \eqref{102a}, $v(t,x)$ is a parameter related to the free surface elevation from equilibrium.

In the past 30 years, the qualitative behaviors of solutions for the CH equation and its various generalizations have been extensively studied, such as local well-posedness \cite{Ji2021,Ji2022,Zhou2022}, global existence of strong solutions
\cite{Constantin1999,Ji2021,Ji2022,Zhou2022}, wave breaking phenomena \cite{Constantin1998,
Constantin1999,Constantin20002,Freire1,Freire2,Freire3,Ji2021,Ji2022,Zhang2023,Zhou2022}, and global existence of weak solutions \cite{Bressan2007,Constantin20004,Wahlen2006,Zhou2024}.
Recently, Brandolese and his co-author \cite{Brandolese2012,Brandolese20141,Brandolese20142,Brandolese20143} obtained a series of interesting results on the blowup problem of CH type equations. Their results show that the local structure of the solution will affect the blowup of the solution. Following the idea of \cite{Brandolese2012,Brandolese20141,Brandolese20142,Brandolese20143}, Chen and his co-author \cite{Chen2016,Chen2017} obtain some local-in-space blowup results for a class of integrable peakon equations and a rotation-two-component Camssa-Holm system;
Novruzov \cite{Novruzov2017,Novruzov2022} establish two local-in-space blowup criteria for a class of nonlinear dispersive wave equations and a two-component nonlinear dispersive wave system.

Inspired by the previous work, the goal of this paper is to establish a new blowup criterion for the Cauchy problem \eqref{101}-\eqref{101a}.

The rest of this paper is organized as follows. In Section 2, we recall some preliminary results. In Section 3, we give a new blowup criterion for the Cauchy problem \eqref{101}-\eqref{101a}.

\textbf{Notation.} Throughout this paper, all spaces of functions are over $\mathbb{S}$ and for simplicity, we drop $\mathbb{S}$ in our notation of function spaces if there is no ambiguity. Additionally, $\|\cdot\|_{s}$ denotes the norm in the Sobolev space $H^{s}(\mathbb{S})$.
\section{Preliminaries}

\setcounter{equation}{0}

\label{sec:2}

In this section, we recall some known results which will be used later. Let $\Lambda=(1-\partial_{x}^{2})^{\frac{1}{2}}$. Then the operator $\Lambda^{-2}$ acting on $L^{2}(\mathbb{S})$ can be expressed by its associated Green's function
$p(x):=\frac{\cosh\left(x-[x]-\frac{1}{2}\right)}{2\sinh \frac{1}{2}},$
where $[x]$ stands for the integer part of $x$ and $\ast$ represents the spacial convolution, as
\begin{equation*}
\Lambda^{-2}f(x)=p\ast f(x)=\frac{1}{2}\int_{0}^{1}\frac{\cosh\left(x-y-[x-y]-\frac{1}{2}\right)}{\sinh \frac{1}{2}}f(y)dy,\ \ f\in L^{2}(\mathbb{S}).
\end{equation*}
Thus, by using the above identities, we can rewrite \eqref{101} as
\begin{eqnarray}
\begin{cases}
u_{t}+(u+v)u_{x}=-P\ast(uv_{x})-\partial_{x}P\ast\left(u^{2}+\frac{1}{2}u_{x}^{2}+u_{x}v_{x}
+\frac{1}{2}v^{2}-\frac{1}{2}v_{x}^{2}\right),\\
v_{t}+(u+v)v_{x}=-P\ast(u_{x}v)-\partial_{x}P\ast\left(v^{2}+\frac{1}{2}v_{x}^{2}+u_{x}v_{x}
+\frac{1}{2}u^{2}-\frac{1}{2}u_{x}^{2}\right).
\end{cases}
\label{201}
\end{eqnarray}

In what follows, we recall the local well-posedness result, conservation law and the sufficient and necessary condition of wave breaking for \eqref{201}.

\begin{lemma}[\cite{Fu2010}]
\label{lemma201}  Given $z_{0}:=\left(
\begin{array}{cc}  u_{0}\\v_{0}\end{array} \right)\in H^{s}\times H^{s} (s>\frac{3}{2})$, then there exists a maximal time $T=T(z_{0})>0$ and a unique solution $z:=\left(
\begin{array}{cc}  u\\v\end{array} \right)$ of the Cauchy problem \eqref{101}-\eqref{101a} such that
$$z=z(\cdot,z_{0})\in C([0,T);H^{s}\times H^{s})\cap C^{1}([0,T);H^{s-1}\times H^{s-1}).$$
Moreover, the solution $z$ depends continuously on the initial data, i.e., the map
$$z_{0}\mapsto z(\cdot,z_{0}):H^{s}\times H^{s}\rightarrow C([0,T);H^{s}\times H^{s})\cap C^{1}([0,T);H^{s-1}\times H^{s-1})$$
is continuous and the maximal time of existence $T>0$ can be chosen to be independent of $s$.
\end{lemma}

\begin{lemma} [\cite{Fu2010}]
\label{lem202}  Let $z_{0}(x)\in H^{s}\times H^{s}\ (s>\frac{3}{2})$ and $T>0$ be the maximal existence time of the corresponding solution $z(t,x)$ of \eqref{201}. Then for any $t\in[0,T)$, we have
$$E(t)=\|u\|^{2}_{1}+\|v\|^{2}_{1}=\|u_{0}\|^{2}_{1}+\|v_{0}\|^{2}_{1}=E(0).$$
\end{lemma}

\begin{lemma} [\cite{Fu2010}]
\label{lem203} Let $z_{0}(x)\in H^{s}\times H^{s}\ (s>\frac{3}{2})$ and $T>0$ be the maximal existence time of the corresponding solution $z(t,x)$ of \eqref{201}. Then $T$ is finite if and only if
\begin{equation*}
\liminf_{t\rightarrow T}\inf_{x\in\mathbb{\mathbb{S}}}{\left(u_{x}(t,x)+v_{x}(t,x)\right)}=-\infty.
\end{equation*}
\end{lemma}

\section{Blowup result}

\setcounter{equation}{0}

\label{sec:3}
In this section, we will give a new blowup criterion for problem \eqref{201}. To this end, we introduce the associated Lagrangian scale of problem \eqref{201}
\begin{equation}
\begin{cases}
\frac{dq(t,x)}{dt}=(u+v)(t,q(t,x)),\ (t,x)\in[0,T)\times\mathbb{R},\\
q(0,x)=x,\ x\in \mathbb{R}.
\end{cases}
\label{301}
\end{equation}
Applying the classical results in the theory of ODEs, we can show that the map $q(t,\cdot)$ is an increasing diffeomorphism of $\mathbb{R}$ with
$$q_{x}(t,x)=e^{\int_{0}^{t}(u_{x}+v_{x})(\tau,q(\tau,x))d\tau}>0,\ (t,x)\in [0,T)\times \mathbb{R}.$$

Differentiating the first and the second equation in \eqref{201} with respect to $x$ and applying the relation $\partial_{x}^{2}p\ast f(x)=p\ast f(x)-f(x)$, we get
\begin{align*}
&u_{tx}+(u+v)u_{xx}+\frac{1}{2}u_{x}^{2}+\frac{1}{2}v_{x}^{2}\nonumber\\
=&u^{2}+\frac{1}{2}v^{2}
-P\ast\left(u^{2}+\frac{1}{2}u_{x}^{2}+u_{x}v_{x}+\frac{1}{2}v^{2}
-\frac{1}{2}v_{x}^{2}\right)-\partial_{x}P\ast(uv_{x})
\end{align*}
and
\begin{align*}
&v_{tx}+(u+v)v_{xx}+\frac{1}{2}u_{x}^{2}+\frac{1}{2}v_{x}^{2}\nonumber\\
=&v^{2}+\frac{1}{2}u^{2}
-P\ast\left(v^{2}+\frac{1}{2}v_{x}^{2}+u_{x}v_{x}+\frac{1}{2}u^{2}
-\frac{1}{2}u_{x}^{2}\right)-\partial_{x}P\ast(u_{x}v).
\end{align*}
Denote $w=u+v$. Then we have
\begin{align}
w_{t}+ww_{x}&=-\partial_{x}P\ast\left(\frac{3}{2}v^{2}+uv+2u_{x}v_{x}+\frac{3}{2}u^{2}\right),
\label{302a}\\
w_{tx}+ww_{xx}+u_{x}^{2}+v_{x}^{2}
&=\frac{3}{2}u^{2}+\frac{3}{2}v^{2}
-P\ast\left(\frac{3}{2}v^{2}+uv+2u_{x}v_{x}+\frac{3}{2}u^{2}\right)+uv.
\label{302}
\end{align}

Let us denote $'$ to be the derivative $\partial_{t}+w\partial_{x}$. Then the dynamics of $w$ and $w_{x}$ along the characteristics $q(t,x)$ are formulated in the following lemma.

\begin{lemma}
\label{lem204}  Let $w_{0}(x)\in H^{s}\ (s>\frac{3}{2})$. Then $w(t,q(t,x))$ and $w_{x}(t,q(t,x))$ satisfying the following integro-differential equations
\begin{align*}
w'(t)=&P_{+}\ast\left(\frac{3}{2}v^{2}+uv+2u_{x}v_{x}+\frac{3}{2}u^{2}\right)
-P_{-}\ast\left(\frac{3}{2}v^{2}+uv+2u_{x}v_{x}+\frac{3}{2}u^{2}\right),\\
w_{x}'(t)=&-u_{x}^{2}-v_{x}^{2}+\frac{3}{2}u^{2}+\frac{3}{2}v^{2}+uv
-P_{+}\ast\left(\frac{3}{2}v^{2}+uv+2u_{x}v_{x}+\frac{3}{2}u^{2}\right)
\nonumber\\
&-P_{-}\ast\left(\frac{3}{2}v^{2}+uv+2u_{x}v_{x}+\frac{3}{2}u^{2}\right).
\end{align*}
\end{lemma}
\begin{proof}
Denote $p(x):=\frac{\cosh\left(x-[x]-\frac{1}{2}\right)}{2\sinh \frac{1}{2}}$ the fundamental solution of operator $(1-\partial_{x}^{2})^{-1}$ on $\mathbb{S}$ and define the two convolution operators $p_{+}$ and $p_{-}$ as
$$P_{+}\ast f(x)=\frac{1}{2}e^{-x}\int_{0}^{x}e^{y}f(y)dy,\ \
P_{-}\ast f(x)=\frac{1}{2}e^{x}\int^{1}_{x}e^{-y}f(y)dy,$$
then $P=P_{+}+P_{-}$ and $P_{x}=P_{-}-P_{+}$.
From \eqref{302}, we get
\begin{align*}
w'(t)=&w_{t}+ww_{x}\nonumber\\
=&-\partial_{x}P\ast\left(\frac{3}{2}v^{2}+uv+2u_{x}v_{x}+\frac{3}{2}u^{2}\right)\nonumber\\
=&P_{+}\ast\left(\frac{3}{2}v^{2}+uv+2u_{x}v_{x}+\frac{3}{2}u^{2}\right)
-P_{-}\ast\left(\frac{3}{2}v^{2}+uv+2u_{x}v_{x}+\frac{3}{2}u^{2}\right),\nonumber\\
w_{x}'(t)=&w_{tx}+ww_{xx}\nonumber\\
=&-u_{x}^{2}-v_{x}^{2}
+\frac{3}{2}u^{2}+\frac{3}{2}v^{2}+uv
-P\ast\left(\frac{3}{2}v^{2}+uv+2u_{x}v_{x}+\frac{3}{2}u^{2}\right)\nonumber\\
=&-u_{x}^{2}-v_{x}^{2}
+\frac{3}{2}u^{2}+\frac{3}{2}v^{2}+uv
-P_{+}\ast\left(\frac{3}{2}v^{2}+uv+2u_{x}v_{x}+\frac{3}{2}u^{2}\right)\nonumber\\
&-P_{-}\ast\left(\frac{3}{2}v^{2}+uv+2u_{x}v_{x}+\frac{3}{2}u^{2}\right).
\end{align*}
This completes the proof of lemma \ref{lem204}.
\end{proof}

In what follows, we give a new blowup criterion for problem \eqref{201}.
\begin{theorem}
\label{the301} Given $z_{0}(x)\in H^{s}\times H^{s}, s>\frac{3}{2}$. Assume that there exists $x_{0}\in \mathbb{S}$ such that
\begin{equation}
u_{0x}(x_{0})+v_{0x}(x_{0})<-|u_{0}(x_{0})+v_{0}(x_{0})|-\sqrt{2}K,
\label{303}
\end{equation}
where $K=\sqrt{\left(\frac{1}{2}+2\coth \frac{1}{2}\right)E(0)}$.
Then the corresponding solution $z$ of \eqref{201} with the initial data $z_{0}(x)$ blows up in finite time with an estimate of the blowup
time $T^{\ast}$ as
$$0<T^{\ast}\leq\frac{1}{\sqrt{2}K}
\ln\frac{\sqrt{(u_{0x}(x_{0})+v_{0x}(x_{0}))^{2}-(u_{0}(x_{0})+v_{0}(x_{0}))^{2}}
+\sqrt{2}K}{\sqrt{(u_{0x}(x_{0})+v_{0x}(x_{0}))^{2}-(u_{0}(x_{0})+v_{0}(x_{0}))^{2}}
-\sqrt{2}K}.$$
Moreover, the blowup point must be inside the interval
$$\left[x_{0}-\frac{1}{\sqrt 2}\sqrt{E(0)}T^{*},
x_{0}+\frac{1}{\sqrt 2}\sqrt{E(0)}T^{*}\right].$$
\end{theorem}

\begin{proof}
Denote $'$ to be the derivative $\partial_{t}+w\partial_{x}$ along the characteristics $q(t,x_{0})$ emanating from $x_{0}$. Below, we track the dynamics of
\begin{align*}
M(t)=&\left(w-w_{x}\right)(t,q(t,x_{0})),\
N(t)=\left(w+w_{x}\right)(t,q(t,x_{0}))
\end{align*}
along the characteristic. Then we get
\begin{align}
M'(t)=&w'-w_{x}'\nonumber\\
=&u_{x}^{2}+v_{x}^{2}
-\frac{3}{2}u^{2}-\frac{3}{2}v^{2}-uv
+P_{+}\ast\left(3u^{2}+2uv+4u_{x}v_{x}+3v^{2}\right)\nonumber\\
\geq&\frac{1}{2}w_{x}^{2}-\frac{1}{2}w^{2}-u^{2}-v^{2}
+P_{+}\ast\left(3u^{2}+2uv+4u_{x}v_{x}+3v^{2}\right)\nonumber\\
=&-\frac{1}{2}MN(t)-u^{2}-v^{2}
+P_{+}\ast\left(3u^{2}+2uv+4u_{x}v_{x}+3v^{2}\right),
\label{304}\\
N'(t)=&w'+w_{x}'\nonumber\\
=&-u_{x}^{2}-v_{x}^{2}
+\frac{3}{2}u^{2}+\frac{3}{2}v^{2}+uv
-P_{-}\ast\left(3u^{2}+2uv+4u_{x}v_{x}+3v^{2}\right)\nonumber\\
\leq&-\frac{1}{2}w_{x}^{2}+\frac{1}{2}w^{2}+u^{2}+v^{2}
-P_{-}\ast\left(3u^{2}+2uv+4u_{x}v_{x}+3v^{2}\right)\nonumber\\
=&\frac{1}{2}MN(t)+u^{2}+v^{2}
-P_{-}\ast\left(3u^{2}+2uv+4u_{x}v_{x}+3v^{2}\right),
\label{305}
\end{align}
where we used $u^{2}+v^{2}\geq \frac{(u+v)^{2}}{2}$ and $u_{x}^{2}+v_{x}^{2}\geq \frac{(u_{x}+v_{x})^{2}}{2}$.

According Sobolev's embedding theorem, Young's inequality and lemma \ref{lem202}, we get
\begin{equation}
\|u^{2}+v^{2}\|_{L^{\infty}}\leq \frac{1}{2}(\|u\|^{2}_{1}+\|v\|^{2}_{1})
=\frac{1}{2}E(0),
\label{306}
\end{equation}
and
\begin{align}
\|P_{+}\ast\left(3u^{2}+2uv+4u_{x}v_{x}+3v^{2}\right)\|_{L^{\infty}}
\leq&4\|P_{+}\|_{L^{\infty}}(\|u^{2}\|_{L^{1}}+\|v^{2}\|_{L^{1}}
+\|u_{x}^{2}\|_{L^{1}}+\|v_{x}^{2}\|_{L^{1}})\nonumber\\
=&4\|P_{+}\|_{L^{\infty}}(\|u\|^{2}_{L^{2}}+\|v\|^{2}_{L^{2}}
+\|u_{x}\|^{2}_{L^{2}}+\|v_{x}\|^{2}_{L^{2}})\nonumber\\
=&2\coth \frac{1}{2}E(0).
\label{307}
\end{align}
Similarly, we get
\begin{equation}
\|P_{-}\ast\left(3u^{2}+2uv+4u_{x}v_{x}+3v^{2}\right)\|_{L^{\infty}}
\leq2\coth \frac{1}{2}E(0).
\label{308}
\end{equation}
Substituting \eqref{306}--\eqref{308} into \eqref{304}--\eqref{305}, we get
\begin{equation}
M'(t)\geq -\frac{1}{2}MN(t)-K^{2}\ \text{and}\
N'(t)\leq \frac{1}{2}MN(t)+K^{2},
\label{3010}
\end{equation}
where $K=\sqrt{\left(\frac{1}{2}+2\coth \frac{1}{2}\right)E(0)}$.
The expected monotonicity of $M(t)$ and $N(t)$ indicate that we would like to have
\begin{equation}
\frac{1}{2}MN(t)+K^{2}<0,\ t\in[0,T).
\label{3011}
\end{equation}
From \eqref{303}, we get
\begin{equation*}
w_{0x}+w_{0}<-\sqrt{2}K\ \text{and}\
w_{0x}-w_{0}<-\sqrt{2}K.
\end{equation*}
Thus, we have
\begin{equation}
M(0)>\sqrt{2}K,\ N(0)<-\sqrt{2}K,\ M'(0)>0,\ N'(0)<0.
\label{3012}
\end{equation}
Therefore, over the time of existence along the characteristics emanating from $x_{0}$ it always holds that
\begin{equation}
M'(t)>0,\ N'(t)<0,\ MN(t)<-2K^{2}.
\label{3013}
\end{equation}

In what follows, we consider the evolution of the quantity $g(t)=\sqrt{-MN(t)}$. Since $M-N\geq 2g$, then we get
\begin{align}
g'(t)=-\frac{M'N+MN'}{2\sqrt{-MN(t)}}
\geq&-\frac{\left(-\frac{1}{2}MN-K^{2}\right)N
+M\left(\frac{1}{2}MN+K^{2}\right)}{2\sqrt{-MN}}\nonumber\\
\geq& -\frac{(M-N)K^{2}+\frac{1}{2}MN(M-N)}{2\sqrt{-MN}}\nonumber\\
=&-\frac{1}{2\sqrt{-MN}}\left(K^{2}+\frac{1}{2}MN\right)(M-N)
\nonumber\\
\geq&\frac{1}{2}g^{2}-K^{2}.
\label{3014}
\end{align}
From \eqref{3013}, we obtain that $g(t)=\sqrt{-MN(t)}$ is increasing when $t\in[0,T)$, thus
\begin{equation*}
g(t)=\sqrt{-MN(t)}\geq \sqrt{-MN(0)}=\sqrt{(w^{2}_{0x}-w^{2}_{0})(x_{0})}>\sqrt{2}K.
\end{equation*}
Therefore, we get
\begin{equation*}
\frac{1}{2}g^{2}-K^{2}>0.
\end{equation*}
From \eqref{3014}, we can obtain that $\lim_{t\rightarrow T^{\ast}}g(t)=+\infty$. Moreover, the maximal existence time $T^{\ast}$ can be estimated from above as follows
\begin{align*}
T^{\ast}&\leq\frac{1}{\sqrt{2}K}
\ln\frac{g(0)+\sqrt{2}K}{g(0)-\sqrt{2}K}\nonumber\\
&=\frac{1}{\sqrt{2}K}
\ln\frac{\sqrt{(u_{0x}(x_{0})+v_{0x}(x_{0}))^{2}-(u_{0}(x_{0})+v_{0}(x_{0}))^{2}}
+\sqrt{2}K}{\sqrt{(u_{0x}(x_{0})+v_{0x}(x_{0}))^{2}-(u_{0}(x_{0})+v_{0}(x_{0}))^{2}}
-\sqrt{2}K}.
\end{align*}
Since $g(t)\leq\frac{M-N}{2}=-w_{x}(t,q(t,x_{0}))$, we get $\lim_{t\rightarrow T^{\ast}}w_{x}(t,q(t,x_{0}))=-\infty$.

From the above argument, we can obtain the information for the location point of blowup. In fact, this information is a consequence of the elementary calculus inequality for continuously derivable functions
$$f: |f(t)- f(0)| \leq \kappa |t|,\ \kappa=\sup |f'|.$$
We can to apply this relation to the function $f(t)=q(t,x_0)$.
From \eqref{301}, we get
$$\kappa\leq \|u(t)\|_{L^{\infty}}+\|v(t)\|_{L^{\infty}}\leq \frac{1}{\sqrt 2}(\|u_0\|_{1}+\|v_0\|_{1})\leq \frac{1}{\sqrt 2}\sqrt{E(0)}.$$
Moreover, $f(0)=x_0$.
So, we deduce, for $0<t<T^{*}$,
$$|q(t,x_0)-x_0|\leq\frac{1}{\sqrt 2}\sqrt{E(0)}t.$$
Hence, the blowup point $q(T^{*},x_0)$ must be inside the interval
$$\left[x_{0}-\frac{1}{\sqrt 2}\sqrt{E(0)}T^{*},
x_{0}+\frac{1}{\sqrt 2}\sqrt{E(0)}T^{*}\right].$$
This completes the proof of theorem \ref{the301}.
\end{proof}



{\bf Acknowledgement.}
This work is partially supported by NSFC Grant (no. 12201539) and Natural Science Foundation of Xinjiang Uygur Autonomous Region (no. 2022D01C65) and Natural Science Foundation of Gansu Province (no. 23JRRG0006) and The Youth Doctoral Support Project for Universities in Gansu Province (no. 2024QB-106).
\bibliographystyle{elsarticle-num}
\bibliography{<your-bib-database>}



\end{document}